\documentclass[10.5pt]{amsart}
\usepackage{amscd,amsfonts,amssymb,amsmath}
\usepackage{hyperref}
\usepackage{epsfig}
\usepackage{mathtools}
\usepackage{tabu}
\usepackage{tikz-cd}
\usepackage{doi}
\usepackage{siunitx}
\newtheorem{theorem}{Theorem}[section]
\newtheorem{corollary}[theorem]{Corollary}
\newtheorem{lemma}[theorem]{Lemma}

\newtheorem{remark}[theorem]{Remark}
\numberwithin{equation}{subsection}
\newtheorem{thmA}{Theorem}

\newtheorem*{ack}{Acknowledgements}
\usepackage[all,cmtip]{xy}
\usepackage{graphicx}

\begin{document}

\title{Minimal presentation, finite quotients and lower central series of cactus groups}

\author{Hugo Chemin}
\author{Neha Nanda}
\address{Laboratoire de Mathématiques Nicolas Oresme UMR CNRS 6139
, Université de Caen Normandie, 14000 Caen, France}
\email{hugo.chemin@unicaen.fr}
\email{nehananda94@gmail.com}


\begin{abstract}
This article deals with the study  of cactus groups from a combinatorial point of view. These groups have been gaining prominence lately in various domains of mathematics, amongst which are their relations with well-known groups such as braid groups, diagram groups, to name a few. We compute a minimal presentation for cactus groups in terms of generators and non-redundant relations.  We also construct homomorphisms of these groups onto certain finite groups, which leads to results about finite quotients of cactus groups. More precisely, we prove that all (infinite) dihedral groups appear as quotients of cactus groups. We also investigate the lower central series and its consecutive quotients. While there are already known established similarities with braid groups, we deduce a considerable disparity between the two groups.
\end{abstract}

\maketitle

\section{Introduction}\label{Introduction}
The cactus group $J_n$ first appeared in the works of Devadoss \cite{MR1718078} and Davis-Januszkiewicz-Scott \cite{MR1985196} under the name of \textit{quasibraid groups} and \textit{mock reflection groups}, respectively. The group $J_n$ is  the fundamental group of the quotient orbifold of $\overline{M}^{n+1}_0(\mathbb{R})$, the Deligne-Knudson-Mumford moduli space of stable real curves of genus $0$ with $n + 1$ marked points, by the action of symmetric group $S_n$ that permutes the first $n$ of those points. The picture of stable real curves in this space, which looks like an Opuntia cactus, hints at why this group is given the name cactus group.\\
The term \textit{quasibraid} in Devadoss' work is due to the resemblance of cactus groups with Artin braid groups, through the machinery of cyclic operads of mosaics which corresponds to cactus groups, just as cube operads correspond to braid groups.  One way to study the space $\overline{M}^{n+1}_0(\mathbb{R})$ is through the iterated blow-ups of braid hyperplane arrangements, which suggests a noteworthy analogy with the pure braid group \cite{MR1985196}. As for braided monoidal categories, coboundary categories have cactus groups which acts on multiple tensor products of objects \cite{MR2219257}. Coboundary categories are then used to study the crystals of reductive Lie algebras of finite dimension and the representations of coboundary Hopf algebras. The cactus group also acts on standard tableaux via the \textit{Schüzenberger involution} which may be recovered as a monodromy action of the cactus group on the simultaneous spectrum of the Gaudin Hamiltonians \cite{MR3845719}. These groups also appear in the literature in the context of hives and octahedron recurrences \cite{MR2261754, MR2097327}, and are used as a tool in representation theory \cite{MR3567259, MR4118639, MR3945740}. \\
Formally, the cactus group $J_n$ is generated by $\{ \sigma_{p,q}, ~ 1 \leq p < q \leq n \}$ with defining relations:
\begin{eqnarray}\label{Canonical-Presentation}
\sigma_{p,q}^{2} &=&1 \hspace*{5mm} \textrm{for } 1 \leq p < q \leq n,\\ 
\sigma_{p,q}\sigma_{r,s} &=& \sigma_{r,s}\sigma_{p,q} \hspace*{5mm} \textrm{for } [p,q] \cap [r,s]= \emptyset,\\
\sigma_{p,q}\sigma_{r,s} &=& \sigma_{p+q-s,p+q-r}\sigma_{p,q} \hspace*{5mm} \textrm{for } [r,s] \subset [p,q].
\end{eqnarray}

Here, $[p,q]=\{p, p+1, \dots , q-1, q\}$. There is a surjective homomorphism of $J_n$ onto the symmetric group $S_n$ given by:
$$\begin{array}{cccccccc}
\pi & : & J_n & \to & S_n \\
& & \sigma_{p,q} & \mapsto & s_{p,q},
\end{array}$$
where $s_{p,q}$ is the permutation in $S_n$ given by: 
$$s_{p,q}(i)=\left\{\begin{array}l
p+q-i \ \ \text{if} \ \ i \in [p,q],\\
i \ \ \text{otherwise.}
\end{array}\right.$$
The kernel of this homomorphism is called the \textit{pure cactus group} of order $n$.\\
From a group-theoretic point of view, the relation of cactus groups with other well-known groups beside braid groups has recently been studied too. It has been shown that the pure cactus group embeds into the diagram group, which is a right-angled Coxeter group, hence it is residually nilpotent \cite{MR3988817}. Bellingeri-Chemin-Lebed \cite{bellingeri2022cactus} explored connections of these groups with Mostovoy's Gauss diagram groups and right-angled Coxeter groups, in particular, they showed that the twin groups inject into the cactus groups. They also proved that the word problem for cactus group is solvable, and the triviality of the center of (pure) cactus group, non-existence of odd torsion in cactus groups, and that pure cactus groups are torsion free. Very recently, the linearity of generalised cactus groups was proved which is constructed by replacing the symmetric group associated with them by a Coxeter group \cite{MR4629741}. These groups are also investigated from a geometric point of view in \cite{genevois2022cactus}.\\

In this paper, we investigate algebraic aspects of cactus groups by first determining a minimal presentation in Section  \ref{Minimal-Presentation-Section}. With the convention that $\sigma_i:=\sigma_{1,i}$ for $i= 2, 3, \dots, n$, we prove the following result. 
\begin{thmA}\label{Minimal-Presentation-Cactus-Rewritten}
The cactus group $J_n$ is generated by  $\{ \sigma_{i} ~|~ i= 2, 3, \dots, n \}$ subject to the following relations:
\begin{eqnarray}
\sigma_{i}^{2} &=&1 \hspace*{2mm} \textrm{for } 2 \leq i \leq n, \label{Relation-One}\\ 
(\sigma_{k}\sigma_{i}\sigma_{k}\sigma_{j})^2&=&1 \hspace*{2mm} \textrm{for } 4 \leq i+j \leq k \leq n, ~~ 2 \leq i \leq j, \label{Relation-Two} \\
\sigma_{k}\sigma_{i+j}\sigma_{j}\sigma_{i+j} &=& \sigma_{k-i}\sigma_{j}\sigma_{k-i}\sigma_{k} \hspace*{2mm} \textrm{for } 3 \leq i+j < k \leq n,~1 \leq i, ~2 \leq j, ~i+j \leq k-i.\label{Relation-Three}
\end{eqnarray}
Further, this presentation is minimal in terms of number of generators. 
\end{thmA}
One of the similarities between cactus and braid groups is their interpretation in terms of intertwined strings on the plane with distinct crossings giving information about the group structure. It is natural to compare their algebraic properties. The natural surjection of the braid group $B_n$ onto the symmetric group can be translated in the case of cactus group, as mentioned previously, which leads to the symmetric group being one of the finite quotient of cactus group.  Thus, it is interesting to find  non-Abelian and non-cyclic homomorphisms onto finite groups (other than the symmetric group) which allows us to study finite quotients of these groups. So far the best known non-cyclic quotient of $B_n$ is $S_n$ which is conjectured by Margalit to be the smallest such quotient \cite{MR4117579}. Very recently, with some obvious exceptions, Kolay \cite{MR4634752} proved that if $G$ is a non-cyclic quotient of $B_n$, then either the order of group $G$ is greater than $n!$, or $G=S_n$. The question of finding finite quotients (or homomorphisms onto finite groups) has been broadened to the setting of mapping class groups of surfaces of finite genus \cite{MR4118626, MR2981051}, the commutator subgroup of braid groups \cite{MR4117579, MR4396925}, (unrestricted) virtual and welded braid groups \cite{MR4510845, MR4625994}, surface braid groups \cite{MR4592549,tan2023smallest}, to name a few.  One of the motivations for studying finite homomorphisms and finite quotients is to distinguish finitely-presented groups via their quotients. In the realm of $3$-dimensional
topology, one relevant question is whether the set of finite quotients of a finitely-generated residually-finite group determines the group itself, up to isomorphism. In more formal terminology, the aim is to obtain a comprehensive understanding of which finitely-generated residually-finite
groups have isomorphic \textit{profinite completions}. We refer to the survey \cite{MR3966805} for more details.\\

In Section \ref{Homomorphisms-Section}, we explore the possible finite quotients of cactus groups. The strategy is to construct explicit homomorphisms onto certain groups, and we obtain the following result.

\begin{thmA}\label{Main-Theorem-Quotients}
All the dihedral groups and the infinite dihedral group are non-trivial quotients of the cactus groups. In particular, the dihedral group $D_4$ is the smallest non-Abelian quotient of the group $J_n$, $n \geq 3$ after $S_3$. Further, there is no upper bound on the order of finite quotients of cactus groups.
\end{thmA}

The infinite dihedral group $\mathbb{Z}_2 \ast \mathbb{Z}_2$ is ``universal" in the sense that if we have a homomorphism from $J_n$ onto a dihedral group $D_m$, then it factors through $\mathbb{Z}_2 \ast \mathbb{Z}_2$. This also leads us to conclude that in the same sense, the cactus groups are closer to the right-angled Coxeter groups than to the braid groups.\\
In Section \ref{Lower-Central-Quotient-Section}, we investigate the lower central series of cactus groups $\{\Gamma_i(J_n)\}_{i \in \mathbb{N}}$ by constructing suitable homomorphisms onto finite groups with long lower central series. The quotients groups of lower central series are important group invariants which are interesting to explore. Cactus groups have long central series compared with that of braid groups whose second and third term coincide. A detailed account of the lower central series of the braid groups and their relatives may be found in \cite{darne2022lower}. In our case we deduce the following.

\begin{thmA}\label{Lower-Central-Series:DoNotStop/Factors}
For all $n \geq 3$,  the lower central series of the group $J_n$ does not stop. Furthermore, 
\begin{itemize}
 \item [(i)] $\Gamma_2(J_n)/\Gamma_3(J_n) \cong \mathbb{Z}_2^{\left\lfloor \frac{n}{2} \right\rfloor}$ and
 \item [(ii)] $\Gamma_3(J_n)/\Gamma_4(J_n) \cong \mathbb{Z}_2^{2\lfloor\frac{n}{2}\rfloor-1}.$
 \end{itemize}
\end{thmA}
Using the above theorem, we compute a presentation of the quotient $J_n/\Gamma_3(J_n)$ and we obtain the following result.
\begin{thmA}\label{Theorem_J_n/Gamma_3}
For $n \geq 3$, the group $J_n/\Gamma_3(J_n)$ has order $2^{{\left\lfloor \frac{n}{2} \right\rfloor}+n-1}$, and it is generated by $\{\sigma_i ~|~ i = 2,3,\ldots,n\}$ subject to the relations:
\begin{eqnarray}
\sigma_{i}^{2} &=&1 \hspace*{2mm} \textrm{for } 2 \leq i \leq n, \label{Relation-One-Quotient}\\ 
{[}\sigma_{i},\sigma_{j}{]} &=&1 \hspace*{2mm} \textrm{for } i < \Bigl \lfloor\frac{n+1}{2} \Bigr\rfloor  ~~ \textrm{or } j \equiv i \pmod {2}, \label{Relation-Two-Quotient} \\
{[}\sigma_{i},\sigma_{j},\sigma_{k}{]} &=&1 \hspace*{2mm}  \textrm{for } 2 \leq i,j,k \leq n ,\label{Relation-Three-Quotient} \\
{[}\sigma_{i},\sigma_{j}{]} &=& {[}\sigma_{i},\sigma_{k}{]} \hspace*{2mm} \textrm{for } 2 \leq i \leq j,k \leq n \textrm{ and } k\equiv j \pmod 2. \label{Relation-Four-Quotient}
\end{eqnarray}
In particular, $J_4/\Gamma_3(J_4) \cong \mathbb{Z}_2^2 \wr \mathbb{Z}_2$ and $J_5/\Gamma_3(J_5) \cong \mathbb{Z}_2 \times (\mathbb{Z}_2^2 \wr \mathbb{Z}_2$).
\end{thmA}
\begin{ack}
The authors are grateful to John Guaschi and Paolo Bellingeri for their mentoring, helpful insights and careful reading of the paper, and to Emmanuel Graff and Jacques Darné for helpful discussions and remarks. The first author has received funding from the Normandy region no. 00123353-22E01371. The second author has received funding from the European Union’s Horizon Europe Research and Innovation programme under the Marie Sklodowska
Curie grant agreement no. 101066588.
\end{ack}
\section{A minimal presentation of Cactus groups}\label{Minimal-Presentation-Section}
Recall the standard presentation of the cactus group from Section \ref{Introduction}. Note that $J_2 \cong \mathbb{Z}_2$ and  $J_3 \cong \mathbb{Z}_2 \ast \mathbb{Z}_2$. The generator $\sigma_{p,q}$ of $J_n$ may be represented as the configuration of $n$ monotonic strings on the plane where the crossing involves the strings $p, p+1, \dots, q$ as shown in Figure \ref{Pic:representation}. An example of an element of $J_5$ is shown in Figure \ref{Pic:cactus_example}, and the relations of $J_n$ are depicted in Figure \ref{Pic:cactus_relations}. One feature that distinguishes $J_n$ from $B_n$ is the presence of torsion and the absence of Artin relations. For example, in the group $J_n$, we do not have the braid relation $\sigma_{1,2}\sigma_{2,3}\sigma_{1,2}=\sigma_{2,3}\sigma_{1,2}\sigma_{2,3}$.\\
\begin{figure}[h]
\begin{center}
\begin{tikzpicture}[line cap=round,line join=round,x=0.7cm,y=.6cm,rounded corners=3pt,line width=1pt]
\draw  (2.,-0.5)-- (2.,-3.5);
\draw  (3.,-0.5)-- (3.,-3.5);
\draw  (5.,-0.5)-- (5.,-3.5);
\draw  (12.,-0.5)-- (12.,-3.5);
\draw  (14.,-0.5)-- (14.,-3.5);
\draw  (15.,-0.5)-- (15.,-3.5);
\draw  (6.,-0.5)-- (6.,-1.4)-- (11.,-2.6)-- (11.,-3.5);
\draw  (7.,-0.5)-- (7.,-1.4)-- (10.,-2.6)-- (10.,-3.5);
\draw  (10.,-0.5)-- (10.,-1.4)-- (7.,-2.6)-- (7.,-3.5);
\draw  (11.,-0.5)-- (11.,-1.4)-- (6.,-2.6)-- (6.,-3.5);
\draw (1.68,0.5) node[anchor=north west] {$1$};
\draw (2.68,0.5) node[anchor=north west] {$2$};
\draw (4.28,0.5) node[anchor=north west] {$p-1$};
\draw (5.68,0.4) node[anchor=north west] {$p$};
\draw (6.28,0.5) node[anchor=north west] {$p+1$};
\draw (9.28,0.5) node[anchor=north west] {$q-1$};
\draw (10.68,0.4) node[anchor=north west] {$q$};
\draw (11.28,0.5) node[anchor=north west] {$q+1$};
\draw (13.24,0.5) node[anchor=north west] {$n-1$};
\draw (14.68,0.4) node[anchor=north west] {$n$};
\draw (8,-0.86) node[anchor=north west] {$\cdots$};
\draw (8,-2.5) node[anchor=north west] {$\cdots$};
\draw (12.5,-1.62) node[anchor=north west] {$\cdots$};
\draw (3.5,-1.64) node[anchor=north west] {$\cdots$};
\bigskip\end{tikzpicture}
\caption{Diagrammatic representation of the element $\sigma_{p,q}$ of $J_n$}
\label{Pic:representation}
\end{center}
\end{figure}
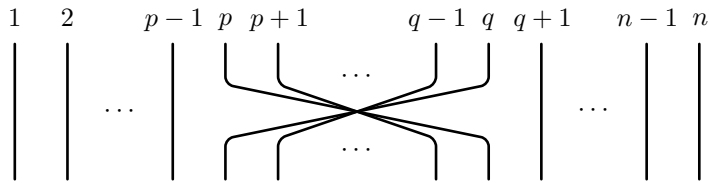
\begin{figure}[h]
\begin{center}
\begin{tikzpicture}[line cap=round,line join=round,x=.5cm,y=.3cm,rounded corners=3pt,line width=1pt]
\draw (1.,-0.5)-- (1.,-1.4)-- (1.,-2.6)-- (1.,-3.4)-- (1.,-4.6)-- (1.,-5.4)-- (3.,-6.6)-- (3.,-7.5);
\draw (2.,-0.5)-- (2.,-1.4)-- (4.,-2.6)-- (4.,-3.4)-- (5.,-4.6)-- (5.,-5.4)-- (5.,-6.6)-- (5.,-7.5);
\draw (3.,-0.5)-- (3.,-1.4)-- (3.,-2.6)-- (3.,-3.4)-- (3.,-4.6)-- (3.,-5.4)-- (1.,-6.6)-- (1.,-7.5);
\draw (4.,-0.5)-- (4.,-1.4)-- (2.,-2.6)-- (2.,-3.4)-- (2.,-4.6)-- (2.,-5.4)-- (2.,-6.6)-- (2.,-7.5);
\draw (5.,-0.5)-- (5.,-1.4)-- (5.,-2.6)-- (5.,-3.4)-- (4.,-4.6)-- (4.,-5.4)-- (4.,-6.6)-- (4.,-7.5);
\bigskip\end{tikzpicture}
\caption{The cactus $\sigma_{2,3}\sigma_{4,5}\sigma_{1,3}$ of $J_5$}
\label{Pic:cactus_example}
\end{center}
\end{figure}
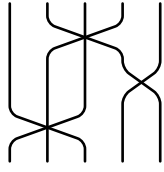
\begin{figure}[h]
\begin{center}
\begin{tikzpicture}[line cap=round,line join=round,x=.5cm,y=.3cm,rounded corners=3pt,line width=1pt]
\draw (2.,-0.5)-- (2.,-1.4)-- (4.,-2.6)-- (4.,-3.4)-- (2.,-4.6)-- (2.,-5.5);
\draw (3.,-0.5)-- (3.,-1.4)-- (3.,-2.6)-- (3.,-3.4)-- (3.,-4.6)-- (3.,-5.5);
\draw (4.,-0.5)-- (4.,-1.4)-- (2.,-2.6)-- (2.,-3.4)-- (4.,-4.6)-- (4.,-5.5);
\draw (1.,-0.5)-- (1.,-1.4)-- (1.,-2.6)-- (1.,-3.4)-- (1.,-4.6)-- (1.,-5.5);
\draw (5.,-0.5)-- (5.,-1.4)-- (5.,-2.6)-- (5.,-3.4)-- (5.,-4.6)-- (5.,-5.5);
\draw (6.,-0.5)-- (6.,-1.4)-- (6.,-2.6)-- (6.,-3.4)-- (6.,-4.6)-- (6.,-5.5);
\node at (7.5,-3) {$=$ \hspace*{3pt}};	
\end{tikzpicture}
\begin{tikzpicture}[line cap=round,line join=round,x=.5cm,y=.3cm,rounded corners=3pt,line width=1pt]
\draw (1.,-0.5)-- (1.,-5.5);
\draw (2.,-0.5)-- (2.,-5.5);
\draw (3.,-0.5)-- (3.,-5.5);
\draw (4.,-0.5)-- (4.,-5.5);
\draw (5.,-0.5)-- (5.,-5.5);
\draw (6.,-0.5)-- (6.,-5.5);
\node at (7.5,-3) {\hspace*{10pt}};
\end{tikzpicture}
\begin{tikzpicture}[line cap=round,line join=round,x=.5cm,y=.3cm,rounded corners=3pt,line width=1pt]
\draw (1.,-0.5)-- (1.,-1.4)-- (2.,-2.6)-- (2.,-3.4)-- (2.,-4.6)-- (2.,-5.5);
\draw (2.,-0.5)-- (2.,-1.4)-- (1.,-2.6)-- (1.,-3.4)-- (1.,-4.6)-- (1.,-5.5);
\draw (3.,-0.5)-- (3.,-1.4)-- (3.,-2.6)-- (3.,-3.4)-- (6.,-4.6)-- (6.,-5.5);
\draw (4.,-0.5)-- (4.,-1.4)-- (4.,-2.6)-- (4.,-3.4)-- (5.,-4.6)-- (5.,-5.5);
\draw (5.,-0.5)-- (5.,-1.4)-- (5.,-2.6)-- (5.,-3.4)-- (4.,-4.6)-- (4.,-5.5);
\draw (6.,-0.5)-- (6.,-1.4)-- (6.,-2.6)-- (6.,-3.4)-- (3.,-4.6)-- (3.,-5.5);
\node at (7.5,-3) {$=$ \hspace*{3pt}};
\end{tikzpicture}
\begin{tikzpicture}[line cap=round,line join=round,x=.5cm,y=.3cm,rounded corners=3pt,line width=1pt]
\draw (6.,-0.5)-- (6.,-1.4)-- (3.,-2.6)-- (3.,-3.4)-- (3.,-4.6)-- (3.,-5.5);
\draw (5.,-0.5)-- (5.,-1.4)-- (4.,-2.6)-- (4.,-3.4)-- (4.,-4.6)-- (4.,-5.5);
\draw (4.,-0.5)-- (4.,-1.4)-- (5.,-2.6)-- (5.,-3.4)-- (5.,-4.6)-- (5.,-5.5);
\draw (3.,-0.5)-- (3.,-1.4)-- (6.,-2.6)-- (6.,-3.4)-- (6.,-4.6)-- (6.,-5.5);
\draw (1.,-0.5)-- (1.,-1.4)-- (1.,-2.6)-- (1.,-3.4)-- (2.,-4.6)-- (2.,-5.5);
\draw (2.,-0.5)-- (2.,-1.4)-- (2.,-2.6)-- (2.,-3.4)-- (1.,-4.6)-- (1.,-5.5);
\end{tikzpicture}
\hspace*{0.25\textwidth}

\begin{tikzpicture}[line cap=round,line join=round,x=.5cm,y=.3cm,rounded corners=3pt,line width=1pt]
\draw (1.,-0.5)-- (1.,-1.4)-- (6.,-2.6)-- (6.,-3.4)-- (6.,-4.6)-- (6.,-5.5);
\draw (2.,-0.5)-- (2.,-1.4)-- (5.,-2.6)-- (5.,-3.4)-- (5.,-4.6)-- (5.,-5.5);
\draw (3.,-0.5)-- (3.,-1.4)-- (4.,-2.6)-- (4.,-3.4)-- (1.,-4.6)-- (1.,-5.5);
\draw (4.,-0.5)-- (4.,-1.4)-- (3.,-2.6)-- (3.,-3.4)-- (2.,-4.6)-- (2.,-5.5);
\draw (5.,-0.5)-- (5.,-1.4)-- (2.,-2.6)-- (2.,-3.4)-- (3.,-4.6)-- (3.,-5.5);
\draw (6.,-0.5)-- (6.,-1.4)-- (1.,-2.6)-- (1.,-3.4)-- (4.,-4.6)-- (4.,-5.5);
\node at (7.5,-3) {$=$ \hspace*{3pt}};
\end{tikzpicture}
\begin{tikzpicture}[line cap=round,line join=round,x=.5cm,y=.3cm,rounded corners=3pt,line width=1pt]
\draw (6.,-0.5)-- (6.,-1.4)-- (3.,-2.6)-- (3.,-3.4)-- (4.,-4.6)-- (4.,-5.5);
\draw (5.,-0.5)-- (5.,-1.4)-- (4.,-2.6)-- (4.,-3.4)-- (3.,-4.6)-- (3.,-5.5);
\draw (4.,-0.5)-- (4.,-1.4)-- (5.,-2.6)-- (5.,-3.4)-- (2.,-4.6)-- (2.,-5.5);
\draw (3.,-0.5)-- (3.,-1.4)-- (6.,-2.6)-- (6.,-3.4)-- (1.,-4.6)-- (1.,-5.5);
\draw (2.,-0.5)-- (2.,-1.4)-- (2.,-2.6)-- (2.,-3.4)-- (5.,-4.6)-- (5.,-5.5);
\draw (1.,-0.5)-- (1.,-1.4)-- (1.,-2.6)-- (1.,-3.4)-- (6.,-4.6)-- (6.,-5.5);
\end{tikzpicture}
\caption{Examples of relations in cactus groups}
\label{Pic:cactus_relations}
\end{center}
\end{figure}
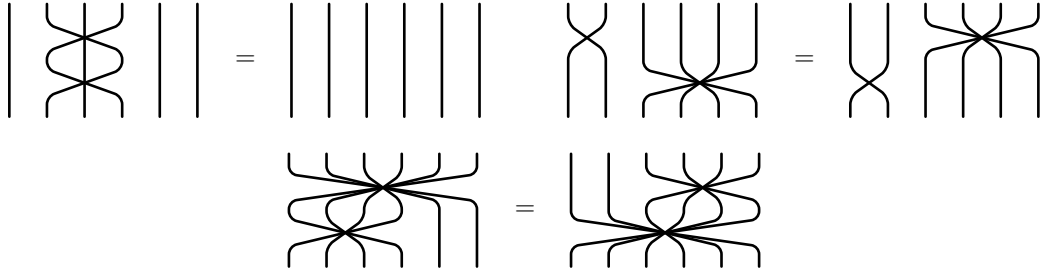

From the relation $\sigma_{p,q}=\sigma_{1,q}\sigma_{1,q-p+1}\sigma_{1,q}$, it is easy to see that $J_n$ may be generated by the set $\{\sigma_{1,i}$, $2 \leq i \leq n\}$. The goal of this section is to obtain a presentation of $J_n$ in terms of these generators (Theorem \ref{Minimal-Presentation-Cactus-Rewritten}).

We start by considering a subset of the above-mentioned set of defining relations of $J_n$. Let $\mathcal{R}$ denote the set of relations with $\sigma_{1,j}$ appearing at least once for every $2 \leq j \leq n-1$. We first prove that we need not consider all the relations in the standard presentation of $J_n$. That is, it suffices to take $\mathcal{R}$ to be the set of defining relations for presenting $J_n$. This leads to the following result. 

\begin{lemma}\label{Reducing-Relations-Lemma}
The standard presentation of the cactus group $J_n$ is equivalent to the presentation:
$$\langle \sigma_{p,q}, ~ 1 \leq p < q \leq n ~|~ \mathcal{R} \rangle,$$
where $\mathcal{R}$ is the subset of relations of the standard presentation in which $\sigma_{1,j}$ appears at least once for $2 \leq j \leq n-1$. 
\end{lemma}

\begin{proof}
We divide the proof into three cases. \\

Case I: Consider $\sigma_{p,q} \sigma_{r,s}=\sigma_{r,s} \sigma_{p,q} $, where $[r,s] \cap [p,q]= \emptyset$, and $p,q,r,s \neq 1$. Without loss of generality, we may assume that $1 < p <q < r < s$. Let us show that $\sigma_{p,q} \sigma_{r,s}=\sigma_{r,s} \sigma_{p,q} $ is a consequence of relations of $\mathcal{R}$:

\begin{eqnarray*}
\sigma_{p,q} \sigma_{r,s} &=& \sigma_{1,q} \sigma_{1, 1+q-p} \underline{\sigma_{1,q} \sigma_{1,s}} \sigma_{1, 1+s-r} \sigma_{1,s}\\
&=& \sigma_{1,q}  \underline{\sigma_{1, 1+q-p}\sigma_{1,s}}\sigma_{1+s-q,s} \sigma_{1, 1+s-r} \sigma_{1,s} ~~ \textrm{ (since } 1+q-p < s)\\
&=&  \underline{\sigma_{1,q} \sigma_{1, s}}\sigma_{s-q+p,s} \sigma_{1+s-q,s} \sigma_{1, 1+s-r} \sigma_{1,s}\\
&=&   \sigma_{1,s} \sigma_{1+s-q, s}\sigma_{s-q+p,s} \underline{\sigma_{1+s-q,s} \sigma_{1, 1+s-r}} \sigma_{1,s}  ~~ \textrm{ (since } [1, 1+s-r] \text{ and } [1+s-q, s] \text{ are disjoint)}\\
&=&   \sigma_{1,s} \sigma_{1+s-q, s} \underline{ \sigma_{s-q+p,s} \sigma_{1, 1+s-r}} \sigma_{1+s-q,s} \sigma_{1,s}  ~~ \textrm{ (since } [1, 1+s-r] \text{ and } [s-q+p, s] \text{ are disjoint)}\\
&=&   \sigma_{1,s} \underline{ \sigma_{1+s-q, s} \sigma_{1, 1+s-r}} \sigma_{s-q+p, s} \sigma_{1+s-q,s} \sigma_{1,s}\\
&=&   \sigma_{1,s} \sigma_{1, 1+s-r}  \sigma_{1+s-q, s} \sigma_{s-q+p, s} \underline{ \sigma_{1+s-q,s} \sigma_{1,s}}
=  \sigma_{1,s} \sigma_{1, 1+s-r}  \sigma_{1+s-q, s} \underline{\sigma_{s-q+p, s}  \sigma_{1,s}} \sigma_{1,q}\\
&=&   \sigma_{1,s} \sigma_{1, 1+s-r}  \underline{\sigma_{1+s-q, s} \sigma_{1,s}} \sigma_{1, 1+q-p} \sigma_{1,q}\\
&=&   \sigma_{1,s} \sigma_{1, 1+s-r} \sigma_{1, s} \sigma_{1,q} \sigma_{1, 1+q-p} \sigma_{1,q}
=   \sigma_{r,s} \sigma_{p,q}.\\
\end{eqnarray*}

Case II: We now consider the relation $\sigma_{p,q} \sigma_{p,r}=\sigma_{p,r} \sigma_{p+r-q,r} $, where $ 1 < p < q< r$.
So we have:
\begin{eqnarray*}
 \sigma_{p,q} \sigma_{1,r}&=&\sigma_{1,r} \sigma_{1+r-q, 1+r-p}\\
 \sigma_{1,q} \sigma_{1,1+q-p}  \sigma_{1,q} \sigma_{1,r}&=&\sigma_{1,r} \sigma_{1, 1+r-p} \sigma_{1, 1+q-p} \sigma_{1, 1+r-p}\\
 \sigma_{1,q} \sigma_{1,1+q-p}  \sigma_{1,q} \sigma_{1,r} \sigma_{1, 1+r-p}&=&\sigma_{1,r} \sigma_{1, 1+r-p} \sigma_{1, 1+q-p}\\
 \sigma_{1,q} \sigma_{1,1+q-p}  \sigma_{1,q} \sigma_{1,r} \sigma_{1, 1+r-p}\sigma_{1,r}&=&\sigma_{1,r} \sigma_{1, 1+r-p} \sigma_{1,r}\sigma_{1,r}\sigma_{1, 1+q-p} \sigma_{1,r}\\
 \sigma_{p,q} \sigma_{p,r}&=&\sigma_{p,r} \sigma_{p+r-q,r}.
\end{eqnarray*}

Case III: Lastly, we consider relations of the form $\sigma_{r,s} \sigma_{p,q}=\sigma_{p,q} \sigma_{p+q-s,p+q-r} $, $[r,s] \subset [p,q]$ and $p >1$. 
 \begin{eqnarray*}
\sigma_{r,s} \sigma_{p,q} &=& \sigma_{1,s} \sigma_{1,1+s-r} \underline{\sigma_{1,s} \sigma_{1,q}} \sigma_{1,1+q-p}  \sigma_{1,q}\\
 &=& \sigma_{1,s} \underline{\sigma_{1,1+s-r} \sigma_{1,q}} \sigma_{1+q-s,q} \sigma_{1,1+q-p}  \sigma_{1,q}\\
 &=& \underline{\sigma_{1,s} \sigma_{1,q}} \sigma_{q+r-s,q} \sigma_{1+q-s,q} \sigma_{1,1+q-p}  \sigma_{1,q}\\
 &=& \sigma_{1,q} \sigma_{1+q-s, q} \underline{ \sigma_{q+r-s,q} \sigma_{1+q-s,q} }\sigma_{1,1+q-p}  \sigma_{1,q}\\
 &=& \sigma_{1,q} \sigma_{1+q-s, q}  \sigma_{1+q-s, q} \sigma_{1+q-s,1+q-r} \sigma_{1,1+q-p}  \sigma_{1,q} ~~ \text{ (by Case II)}\\
 &=& \sigma_{1,q} \underline{\sigma_{1+q-s,1+q-r} \sigma_{1,1+q-p}} \sigma_{1,q}
 = \sigma_{1,q} \sigma_{1,1+q-p}  \underline{ \sigma_{1+r-p,1+s-p}\sigma_{1,q}}\\
 &=& \sigma_{1,q} \sigma_{1,1+q-p}\sigma_{1,q}  \sigma_{q+p-s,q+p-r}
 = \sigma_{p,q} \sigma_{p+q-s,p+q-r}.
 \end{eqnarray*}
\end{proof}
In what follows, we show that $\{\sigma_{1,2},\ldots,\sigma_{1,n}\}$ is a minimal set of generators for $J_n$, and we compute a presentation of $J_n$ in terms of these generators.
\begin{lemma}\label{lemma1}
For all $n \geq 2$, the Abelianisation $J_n/[J_n,J_n]$ of $J_n$ is isomorphic to $\mathbb{Z}_2 ^{n-1}.$
\end{lemma}
\begin{proof}
A presentation of $J_n/[J_n,J_n]$ is given by: $$\begin{array}{ll}
J_n/[J_n,J_n]=\langle
\overline{\sigma_{p,q}}, \ \ 1 \leq p < q \leq n~|~&\overline{\sigma_{p,q}}^2=e, \ \ \overline{\sigma_{p,q}\sigma_{r,s}}=\overline{\sigma_{r,s}\sigma_{p,q}} \ \ \text{for} \ \ [ p,q ] \cap [ r,s ] = \emptyset, \\& \overline{\sigma_{p,q}\sigma_{r,s}}=\overline{\sigma_{p+q-s,p+q-r}\sigma_{p,q}} \ \ \text{for} \ \ [ r,s ] \subset [ p,q ], \\ & \overline{g_1g_2}=\overline{g_2g_1} \ \ \text{for all } (g_1,g_2) \in J_n^2\rangle.
\end{array}$$
Since $\left\{\sigma_{p,q}\right\}_{1 \leq p<q \leq n}$ generates $J_n$, we have $\overline{g_1}\overline{g_2}=\overline{g_2}\overline{g_1}$ for all $(g_1,g_2) \in J_n^2$ if and only if $\overline{\sigma_{p,q}\sigma_{r,s}}=\overline{\sigma_{r,s}\sigma_{p,q}}$ for all $1 \leq p < q \leq n$ and $1 \leq r < s \leq n$. So, 
$$\begin{array}{ll}
J_n/[J_n,J_n]=\langle
\overline{\sigma_{p,q}}, \ \ 1 \leq p < q \leq n~|~&\overline{\sigma_{p,q}}^2=e, \\& \overline{\sigma_{p,q}\sigma_{r,s}}=\overline{\sigma_{p+q-s,p+q-r}\sigma_{p,q}} \ \ \text{for} \ \ [ r,s ] \subset [ p,q ], \\ & \overline{\sigma_{p,q}\sigma_{r,s}}=\overline{\sigma_{r,s}\sigma_{p,q}}\rangle.
\end{array}$$
Then, it is easy to see that $\overline{\sigma_{p,q}}=\overline{\sigma_{r,m}}$ if and only if $m-r=q-p$. so we have: $$
J_n/[J_n,J_n]=\langle
\overline{\sigma_{1,q}}, \ \ 1 < q \leq n~|~\overline{\sigma_{1,q}}^2=e, \ \ \overline{\sigma_{1,q}\sigma_{1,s}}=\overline{\sigma_{1,s}\sigma_{1,q}}\rangle\cong \left(\mathbb{Z}/2 \mathbb{Z} \right)^{n-1}$$
as required.
\end{proof}
We now prove the main theorem of this section.

\begin{proof}[Proof of Theorem \ref{Minimal-Presentation-Cactus-Rewritten}]
We first prove that relations (\ref{Relation-One})-(\ref{Relation-Three}) are indeed relations in $J_n$. It is evident that $\sigma_{1,i}^2=1$ for all $i \in [2,n]$, which yields Relation \ref{Relation-One}.\\
If $4 \leq i+j \leq k\leq n$ and $2 \leq i \leq j$, we have:
$$\left(\sigma_{1,k}\sigma_{1,i}\sigma_{1,k}\sigma_{1,j}\right)^2 = \left(\sigma_{k-i+1,k}\sigma_{1,j}\right)^2.$$
Now $i+j \leq k$ implies that $j < k-i+1$. Thus, $\sigma_{k-i+1,k}$ and $\sigma_{1,j}$ commute and we have:
$$\left(\sigma_{1,k}\sigma_{1,i}\sigma_{1,k}\sigma_{1,j}\right)^2 = 1.$$
Let $ 3 \leq i+j < k \leq n$ such that $1 \leq i$ and $2 \leq j$, then we have: $$\sigma_{1,k}\sigma_{1,i+j}\sigma_{1,j}\sigma_{1,i+j} =\sigma_{1,k}\sigma_{i+1,i+j}.$$
On the other hand, for $k-i > j$, we have $k-i-j+1 >1$ and $$\sigma_{1,k-i}\sigma_{1,j}\sigma_{1,k-i}\sigma_{1,k}=\sigma_{k-i-j+1,k-i}\sigma_{1,k}=\sigma_{1,k}\sigma_{i+1,i+j}.$$
Next, we show that each of the relations in the standard presentation can be transformed into the defining relations given in the new presentation by Tietze transformations. We recall that for  $1\leq p< q \leq n$, we have $\sigma_{p,q}=\sigma_{1,q}\sigma_{1,q-p+1}\sigma_{1,q}$.\\
For $1\leq p< q \leq n$, we have $q-p+1 \in [ 2,n]$  and

$$\begin{array}{llllllllll}
 & \left(\sigma_{p,q}\right)^2=1
\Leftrightarrow & \left(\sigma_{1,q}\sigma_{1,q-p+1}\sigma_{1,q}\right)^2=1
\Leftrightarrow & \left(\sigma_{1,q-p+1}\right)^2=1.
\end{array}$$

Next, let $1 \leq r < s \leq n$ and $1\leq p < q \leq n$ such that $[ r,s] \cap [ p,q ] =\emptyset$. Without loss of generality, we may assume that $s < p$. Then:
$$\begin{array}{lllll}
& \left(\sigma_{r,s}\sigma_{p,q}\right)^2 =1
\Leftrightarrow \left(\sigma_{1,s}\sigma_{1,s-r+1}\underline{\sigma_{1,s}\sigma_{1,q}}\sigma_{1,q-p+1}\sigma_{1,q}\right)^2 =1\\
\Leftrightarrow & \left(\sigma_{1,s}\sigma_{1,s-r+1}\sigma_{1,q}\underline{\sigma_{q+1-s,q}\sigma_{1,q-p+1}}\sigma_{1,q}\right)^2 =1 ~~ \text{ since } [ 1,q+1-p] \cap [ q+1-s,q] = \emptyset\\
\Leftrightarrow & \left(\sigma_{1,s}\sigma_{1,s-r+1}\sigma_{1,q}\sigma_{1,q-p+1}\underline{\sigma_{q+1-s,q}\sigma_{1,q}}\right)^2 =1\\
\Leftrightarrow & \left(\sigma_{1,s}\sigma_{1,s-r+1}\sigma_{1,q}\sigma_{1,q-p+1}\sigma_{1,q}\sigma_{1,s}\right)^2 =1\\
\Leftrightarrow & \left(\sigma_{1,s-r+1}\sigma_{1,q}\sigma_{1,q-p+1}\sigma_{1,q}\right)^2 =1.\\
\end{array}$$ 

Now we suppose that $i=s-r+1$, $j=q-p+1$ and $k=q$. From the way we have defined $[r,s]$ and $[p,q]$, we may assume that $i \leq j$. We have $s < p$ and $r \geq 1$, so that  $i+j=q-p+s-r+2\leq q$. If not, that is, if $i+j >q$, then $p < s-r+2 <s+1$, which is a contradiction. Moreover, since $q-p  \geq 1$ and $s-r \geq 1$, we have $i+j=q-p+s-r+2 \geq 4$. Hence, we obtain the relation 
$(\sigma_{1,k}\sigma_{1,i}\sigma_{1,k}\sigma_{1,j})^2=1$
for $4 \leq i+j \leq k \leq n$ and  $2 \leq i \leq j$, which yields (\ref{Relation-Two}). \\

We now consider the last set of relations $1 \leq r < s \leq n$ and $1\leq p < q \leq n$ such that $[r,s] \subset [p,q]$. We have:
$$\begin{array}{llllllll}
& \sigma_{p,q}\sigma_{r,s}=\sigma_{p+q-s,p+q-r}\sigma_{p,q}\\
\Leftrightarrow & \sigma_{1,q}\sigma_{1,q-p+1}\underline{\sigma_{1,q}\sigma_{1,s}}\sigma_{1,s-r+1}\sigma_{1,s}=\sigma_{1,p+q-r}\sigma_{1,s-r+1}\sigma_{1,p+q-r}\sigma_{1,q}\sigma_{1,q-p+1}\sigma_{1,q}\\
\Leftrightarrow & \sigma_{1,q}\sigma_{1,q-p+1}\sigma_{q-s+1,q}\underline{\sigma_{1,q}\sigma_{1,s-r+1}}\sigma_{1,s}=\sigma_{1,p+q-r}\sigma_{1,s-r+1}\sigma_{1,p+q-r}\sigma_{1,q}\sigma_{1,q-p+1}\sigma_{1,q}\\
\Leftrightarrow & \sigma_{1,q}\sigma_{1,q-p+1}\sigma_{q-s+1,q}\sigma_{q+r-s,q}\underline{\sigma_{1,q}\sigma_{1,s}}=\sigma_{1,p+q-r}\sigma_{1,s-r+1}\sigma_{1,p+q-r}\sigma_{1,q}\sigma_{1,q-p+1}\sigma_{1,q}\\
\Leftrightarrow & \sigma_{1,q}\sigma_{1,q-p+1}\sigma_{q-s+1,q}\sigma_{q-s+r,q}\sigma_{q-s+1,q}\sigma_{1,q}=\sigma_{1,p+q-r}\sigma_{1,s-r+1}\sigma_{1,p+q-r}\sigma_{1,q}\sigma_{1,q-p+1}\sigma_{1,q}\\
\Leftrightarrow & \sigma_{1,q}\sigma_{1,q-p+1}\sigma_{q-s+1,q}\sigma_{q-s+r,q}\sigma_{q-s+1,q}=\sigma_{1,p+q-r}\sigma_{1,s-r+1}\sigma_{1,p+q-r}\sigma_{1,q}\sigma_{1,q-p+1}.
\end{array}$$

Now, $r \geq 1$ and $s >p$, therefore $[ q-s+r,q] \subset [ q-s+1,q ]$ and $[ q-s+1, q-r+1] \subset [1,q-p+1]$. So we obtain:
$$\begin{array}{llllllll}
& \sigma_{p,q}\sigma_{r,s}=\sigma_{p+q-s,p+q-r}\sigma_{p,q}\\
\Leftrightarrow & \sigma_{1,q}\underline{\sigma_{1,q-p+1}\sigma_{q-s+1,q-r+1}}\sigma_{q-s+1,q}\sigma_{q-s+1,q}=\sigma_{1,p+q-r}\sigma_{1,s-r+1}\sigma_{1,p+q-r}\sigma_{1,q}\sigma_{1,q-p+1}\\
\Leftrightarrow & \sigma_{1,q}\sigma_{p-r+1,p-s+1}\sigma_{1,q-p+1}=\sigma_{1,p+q-r}\sigma_{1,s-r+1}\sigma_{1,p+q-r}\sigma_{1,q}\sigma_{1,q-p+1}\\
\Leftrightarrow & \sigma_{1,q}\sigma_{1,p-s+1}\sigma_{1,s-r+1}\sigma_{1,p-s+1}=\sigma_{1,p+q-r}\sigma_{1,s-r+1}\sigma_{1,p+q-r}\sigma_{1,q}.
\end{array}$$
We set $i=r-p$, $j=s-r+1$ and $k=q$. The case $i=0$ is trivial and does not yield any non-trivial relation, so we discard this case.\\
Considering $i \geq 1$, we have $i+j=s-p+1 \leq s \leq q=k$. But the case $i+j=k$ corresponds to the relation $\sigma_{1,j}\sigma_{1,k}=\sigma_{1,j}\sigma_{1,k}$ which is trivial, so we can remove it from the presentation.\\
Lastly, we notice that the relations corresponding to $k-i > i+j$ coincide with the relations $k-i < i+j.$ To see this, it suffices to suppose that $i+j \leq k-i$. It is easy to check that the relations corresponding to $k-i \geq i+j$ are all distinct. Also, since $r < s$, this implies that $j\geq 2$. So we obtain: $$\sigma_{1,k}\sigma_{1,i+j}\sigma_{1,j}\sigma_{1,i+j}=\sigma_{1,k-i}\sigma_{1,j}\sigma_{1,k-i}\sigma_{1,k},$$
for $3 \leq i+j < k \leq n,~1 \leq i, ~2 \leq j$ and $i+j \leq k-i$.\\

The fact that the number of generators is minimal is due to the fact that $J_n/[J_n,J_n] \cong \mathbb{Z}_2 ^{n-1}$ proved in Lemma \ref{lemma1}.
\end{proof}

\begin{remark}
We may compute and compare the number of relations in the two presentations. Let $G_n$ (resp. $\tilde{G}_n$) be the number of generators and $R_n$ (resp. $\tilde{R}_n$) be the number of relations in the standard (resp. new) presentation of the group $J_n$. Then:
$$\begin{array}{ccc}
R_n=\frac{6n^4-16n^3+48n^2-32n-3+3(-1)^n}{96} & \text{and} & G_n= \binom{n}{2}
\end{array}$$ and in the new presentation, we have: $$\begin{array}{ccc}
\tilde{R}_n=\frac{4n^3-18n^2+44n-27+3(-1)^n}{24} & \text{and} & \tilde{G}_n= n-1.
\end{array}$$
So in the original presentation, the number of generators is equivalent to $n^2$ and the number of relations is equivalent to $n^4/16$. On the other hand, in the new presentation, the number of generators is equivalent to $n$ and the number of relations is equivalent to $n^3/6$. One may show also that the relations obtained from the new presentation are non-redundant and distinct. We suspect that $\tilde{R}_n$ is the minimal number of relations needed to define the group $J_n$ but we do not yet have a formal proof.
\end{remark}

In the subsequent sections, we will make use of the new presentation to obtain some algebraic properties of $J_n$. For convenience, we omit ``1" in the notation of the generator $\sigma_{1,i}$ from now. 

\section{From cactus groups to dihedral groups}\label{Homomorphisms-Section}
In this section, using the new presentation of Theorem \ref{Minimal-Presentation-Cactus-Rewritten}, we construct explicit surjective homomorphisms of cactus groups onto (infinite) dihedral groups. Our aim is to prove Theorem \ref{Main-Theorem-Quotients}.\\
Consider the dihedral group $D_n$ of order $2n$ with the presentation:
$$ \langle a,b |a^2=b^2=(ab)^n=1\rangle.$$
\begin{theorem}\label{Map-to-D4}
For $n \geq 3$, there exists a surjective homomorphism $\varphi: J_n \to D_4$ given by 
 $$\begin{array}{ccccccc}
 \sigma_{i} & \mapsto & \left\{\begin{array}{lll}
1 & \text{if} & i < \left\lfloor \frac{n+1}{2} \right\rfloor \\
a & \text{if} & i \equiv 0 \pmod 2 \ \ \text{and} \ \ i \geq \left\lfloor \frac{n+1}{2} \right\rfloor \\
b & \text{if} & i \equiv 1 \pmod 2 \ \ \text{and} \ \ i \geq \left\lfloor \frac{n+1}{2} \right\rfloor,
\end{array}\right.
\end{array}$$
where $i=2, 3, \dots, n.$
\end{theorem}
\begin{proof}
We check that $\varphi$ satisfies the relations given in Theorem \ref{Minimal-Presentation-Cactus-Rewritten}.\\
By definition, $\varphi(\sigma_{i})^2=1$, for all $i= 2, 3, \dots, n$.\\
We now consider the relations of type $(\sigma_{k}\sigma_{j}\sigma_{k}\sigma_{i})^2=1$, with $4 \leq i+j \leq k \leq n$ and $2\leq i \leq j$.\\
If $j < \left\lfloor \frac{n+1}{2} \right\rfloor$, then since $i \leq j$, we have $i < \left\lfloor \frac{n+1}{2} \right\rfloor$, and we obtain:
$$\varphi((\sigma_{k}\sigma_{j}\sigma_{k}\sigma_{i})^2)=\varphi(\sigma_{k}\sigma_{k})^2 =1.$$

Next, suppose that $j \geq \left\lfloor \frac{n+1}{2} \right\rfloor$. Since $i+j \leq n$, we have $i \leq \left\lfloor \frac{n+1}{2} \right\rfloor$. If $i < \left\lfloor \frac{n+1}{2} \right\rfloor$, then: 
$$\varphi((\sigma_{k}\sigma_{j}\sigma_{k}\sigma_{i})^2)=\varphi((\sigma_{k}\sigma_{j}\sigma_{k})^2) =1.$$

Now, if $i = \left\lfloor \frac{n+1}{2} \right\rfloor$, we consider two subcases. Suppose that $n$ is even. Then by our assumption, $i=j=\frac{n}{2}$, and $k=n$. According as $n/2$ is even or odd, we obtain:
$$\varphi((\sigma_{k}\sigma_{j}\sigma_{k}\sigma_{i})^2) = \left\{\begin{array}{lll}
a^8 = 1 & \text{if} & n \equiv 0 \pmod 4 \\
(abab)^2 = 1 & \text{if} & n \equiv 2 \pmod 4.
\end{array}\right.$$
The case where $n$ is odd does not satisfy the conditions of Relation (\ref{Relation-Three}), and so is not required. So, the Relation (\ref{Relation-Three}) is preserved under the map $\varphi$. 

Now, we consider the relations of type $$\sigma_{k}\sigma_{i+j}\sigma_{j}\sigma_{i+j} \sigma_{k}\sigma_{k-i}\sigma_{j}\sigma_{k-i}=1, $$ with 
$3 \leq i+j < k \leq n$,  $i+j \leq k-i$ ,$1 \leq i$ and $2 \leq j$.\\
We first suppose that $k < \left\lfloor \frac{n+1}{2} \right\rfloor$. By the bounds on the indices, we get 
$k-i, i+j, j < \left\lfloor \frac{n+1}{2} \right\rfloor$. So, the above relation is preserved trivially by map $\varphi$. So assume that $k \geq \left\lfloor \frac{n+1}{2} \right\rfloor$. We consider two subcases.\\
Let $k \geq \left\lfloor \frac{n+1}{2} \right\rfloor$ and $j < \left\lfloor \frac{n+1}{2} \right\rfloor$. Then:

$$\varphi(\sigma_{k}\sigma_{i+j}\sigma_{j}\sigma_{i+j} \sigma_{k}\sigma_{k-i}\sigma_{j}\sigma_{k-i})= \varphi(\sigma_{k})\varphi(\sigma_{i+j})^2 \varphi(\sigma_{k})\varphi(\sigma_{k-i})^2=  \varphi(\sigma_{k})^2=1.$$

Lastly, we assume that $k \geq \left\lfloor \frac{n+1}{2} \right\rfloor$ and $j \geq \left\lfloor \frac{n+1}{2} \right\rfloor$. Then $k-i>j\geq \left\lfloor \frac{n+1}{2} \right\rfloor$, and

$$\varphi(\sigma_{k}\sigma_{i+j}\sigma_{j}\sigma_{i+j} \sigma_{k}\sigma_{k-i}\sigma_{j}\sigma_{k-i}) = \left\{\begin{array}{lll}
a^8 = 1 & \text{if} & (k,i,j) \equiv (0,0,0) \pmod 2  \\
abbbaaba = 1 & \text{if} & (k,i,j) \equiv (0,0,1) \pmod 2 \\
(ab)^4 = 1 & \text{if} & (k,i,j) \equiv (0,1,0) \pmod 2  \\
baaabbab = 1 & \text{if} & (k,i,j) \equiv (1,0,0) \pmod 2  \\
aabaabbb = 1 & \text{if} & (k,i,j) \equiv (0,1,1) \pmod 2  \\
b^8 = 1 & \text{if} & (k,i,j) \equiv (1,0,1) \pmod 2  \\
bbabbaaa = 1 & \text{if} & (k,i,j) \equiv (1,1,0) \pmod 2  \\
(ba)^4 = 1 & \text{if} & (k,i,j) \equiv (1,1,1) \pmod 2.
\end{array}\right.$$
Hence, the map $\varphi$ is a surjective homomorphism of $J_n$ onto $D_4$.
\end{proof}

Note that the above-mentioned homomorphism $\varphi$ crucially use the relation $(ab)^4=1$ of $D_4$, therefore, it doesn't work for the dihedral group $D_8$, so in the next result we construct a new map onto $D_8$.

\begin{theorem}\label{Map-to-D8}
For $n \geq 2$, there exists a homomorphism $\psi: J_{2n-1} \longrightarrow D_8$ defined on the generators of the new presentation by:
$$\begin{array}{cccccccccccc}
& & \sigma_i & \mapsto & \left\{\begin{array}{llllll}
a & \text{if} \ \ i=n \\
b & \text{if} \ \ i\geq n+1 \ \ \text{and} \ \ i \equiv n+1 \pmod 2 \\
1 & \textrm{otherwise,}
\end{array}\right.
\end{array}$$
where $i=2, 3, \dots, n$. In particular, we obtain a surjective homomorphism of $J_n$ onto $D_8$ via the surjection $J_{n}$ onto $J_{n-1}$ given by: 
$$\begin{array}{cccclcccccc}
q_n & : & J_n & \to & J_{n-1} \\
& & \sigma_i & \mapsto & \left\{\begin{array}{cl}
\sigma_{i-1} & \text{if } i > 2\\
1 & \text{otherwise.}
\end{array}\right.
\end{array}$$
\end{theorem}
\begin{proof}
The relations of the form $\sigma_i^2=1$ are trivially respected by $\psi$. We now consider the relations of type $(\sigma_k\sigma_i\sigma_j\sigma_k)^2=1$ with $4 \leq i+j \leq k \leq 2n-1, ~ 2 \leq i \leq j$. Notice that if $j \geq n$, then $i \leq n-1$, so we have: $$\psi((\sigma_k\sigma_i\sigma_j\sigma_k)^2)=\left\{\begin{array}{lll}
(bba)^2=1 & \text{if} \ \ j=n \ \ \text{and} \ \ k \equiv n+1 \pmod 2\\
a^2=1 & \text{if} \ \ j=n\ \ \text{and} \ \ k \equiv n \pmod 2\\
b^4=1 & \text{if} \ \ j<n \leq k \ \ \text{and} \ \ k \equiv n+1 \pmod 2\\
1 & \text{if} \ \ j<n \leq k \ \ \text{and} \ \ k \equiv n \pmod 2\\
1 & \text{if} \ \ k< n.
\end{array}\right.$$
Lastly, we consider the relations of type $\sigma_k\sigma_{k-i}\sigma_j\sigma_{k-i}\sigma_k\sigma_{i+j}\sigma_j\sigma_{i+j}=1$.
Evidently, if $k<n$ then the relation is preserved trivially under $\psi$.\\
If $k=n$, then $i+j,k-i,j<n$ and again the relation holds.\\
Suppose now that $k=n+1$. Then we have $j<k-i\leq n$, and we obtain: $$\psi(\sigma_k\sigma_{k-i}\sigma_j\sigma_{k-i}\sigma_k\sigma_{i+j}\sigma_j\sigma_{i+j})=b\psi(\sigma_{i+j})^2 b\psi(\sigma_{k-i})^2=1.$$
Finally, suppose that $k >n$. \\
If $j < n$, then as above, $$\psi(\sigma_k\sigma_{k-i}\sigma_j\sigma_{k-i}\sigma_k\sigma_{i+j}\sigma_j\sigma_{i+j})=\psi(\sigma_k)\psi(\sigma_{i+j})^2\psi(\sigma_k)\psi(\sigma_{k-i})^2=\psi(\sigma_k)^2=1.$$
If $j>n$, then $k,i+j,k-i>n$, and $\psi(\sigma_j)$ $\psi(\sigma_k)$, $\psi(\sigma_{i+j})$ and $\psi(\sigma_{k-i})$ are in $\{1,b\}$ and appear an even number of times, so the relation is preserved by $\psi$.\\
Finally, suppose that $j=n$, which implies that $k-i > n$. We have two cases, $i=1$ and $i>1$.\\
If $i=1$, we have:
$$\psi(\sigma_k\sigma_{k-1}\sigma_{i_m}\sigma_{k-1}\sigma_k\sigma_{i_k+1}\sigma_{i_m}\sigma_{i_m+1})=\left\{\begin{array}{lll}
(bab)^2=1 & \text{if} \ \ k \equiv n \pmod 2\\
(bab)^2=1 & \text{if} \ \ k \equiv n+1 \pmod 2.\\
\end{array}\right.$$
If $i>1$, we have:
$$\psi(\sigma_k\sigma_{k-i}\sigma_{i_k}\sigma_{k-i}\sigma_k\sigma_{i_k+i}\sigma_{i_k}\sigma_{i_k+i})=\left\{\begin{array}{lll}
a^2=1 & \text{if} \ \ (k,i) \equiv (n,0) \pmod 2\\
(bba)^2=1 & \text{if} \ \ (k,i) \equiv (n+1,0) \pmod 2\\
(bab)^2=1 & \text{if} \ \ (k,i) \equiv (n,1) \pmod 2\\
(bab)^2=1 & \text{if} \ \ (k,i) \equiv (n+1,1) \pmod 2.
\end{array}\right.$$
Hence, $\psi$ is a homomorphism of $J_n$ onto $D_8$.
\end{proof}

\begin{remark}
The homomorphism $\psi: J_{n} \to D_8$ is well defined if $n \not\equiv 2 \pmod 4$. The obstruction in the case $n \equiv 2 \pmod 4$ is due to the relation $(\sigma_n\sigma_{n/2})^2=1$.
\end{remark}

\begin{remark}\label{D8-not-needed}
Note that in the proof of Theorem \ref{Map-to-D8}, the relation $(ab)^8=1$ of $D_8$ that was required in the proof of Theorem \ref{Map-to-D4} is not needed here. 
\end{remark}
In the following result, we construct another homomorphism of $J_n$ onto $\mathbb{Z}_2 \ast \mathbb{Z}_2$ which will be crucial for the subsequent section. Consider the infinite dihedral group $\mathbb{Z}_2 \ast \mathbb{Z}_2$ with the following presentation:
$$\langle a, b ~|~ a^2=b^2=1 \rangle .$$
\begin{theorem}\label{Map-to-Dinfty}
There exists a surjective homomorphism $\phi: J_n \longrightarrow \mathbb{Z}_2 \ast \mathbb{Z}_2$ given by:
$$\begin{array}{cccccccc}
& & \sigma_i & \mapsto & \left\{\begin{array}{lll}
a(ab)^{n-i} & \text{if} \ \  i \geq \left\lfloor  \frac{n+1}{2}  \right\rfloor\\
1 & \text{otherwise,}
\end{array}\right.
\end{array}$$ where $i = 2,3,\ldots,n.$
\end{theorem}

\begin{proof}
First, we check the images by $\phi$ of the relations of type $\sigma_i^2=1$. For $i < \left\lfloor  \frac{n+1}{2}  \right\rfloor$ we have $\phi(\sigma_i)^2=1$ by definition, and for $i \geq \left\lfloor  \frac{n+1}{2}  \right\rfloor$, we have:
$$\phi(\sigma_i)^2= a(ab)^{n-i}a(ab)^{n-i}=(ba)^{n-i}(ab)^{n-i}=1.$$
Then we check the image by $\phi$ of the relations of type $(\sigma_k\sigma_i\sigma_j\sigma_k)^2=1$. Without loss of generality, we may suppose that $i \geq j$. By the conditions on the relations given by Theorem \ref{Minimal-Presentation-Cactus-Rewritten}, we know that if $i \geq \left\lfloor  \frac{n+1}{2}  \right\rfloor$, then $j < \left\lfloor  \frac{n+1}{2}  \right\rfloor$. Then: 
$$\phi((\sigma_k\sigma_i\sigma_k\sigma_j)^2)=\left\{\begin{array}{lllll}
(a(ab)^{n-k}a(ab)^{n-i}a(ab)^{n-k})^2=1 & \text{if} \ \ \left\lfloor  \frac{n+1}{2}  \right\rfloor \leq i \\
(a(ab)^{n-k})^4=1 & \text{if} \ \   i<\left\lfloor\frac{n+1}{2}  \right\rfloor \leq k \\
1 & \text{if} \ \ k<\left\lfloor\frac{n+1}{2}  \right\rfloor.
\end{array}\right.$$
Finally, we check the image by $\phi$ of the relations of type $\sigma_k\sigma_{k-i}\sigma_j\sigma_{k-i}\sigma_k\sigma_{i+j}\sigma_j\sigma_{i+j}=1$. Without loss of generality, we may suppose that $k-i \geq i+j$.\\
It is obvious that if $k < \left\lfloor\frac{n+1}{2}  \right\rfloor$ or $k-i < \left\lfloor\frac{n+1}{2}  \right\rfloor$, then the above relation is preserved under the map $\phi$ .\\
If $i+j < \left\lfloor\frac{n+1}{2}  \right\rfloor \leq k-i$, then: $$\begin{array}{lcl}
\phi(\sigma_k\sigma_{k-i}\sigma_j\sigma_{k-i}\sigma_k\sigma_{i+j}\sigma_j\sigma_{i+j}) & = & a(ab)^{n-k}a(ab)^{n+i-k}a(ab)^{n+i-k}a(ab)^{n-k}\\
& = & (ba)^{n-k}(ab)^{n+i-k}(ba)^{n+i-k}(ab)^{n-k} =  1.
\end{array}$$
If $j < \left\lfloor\frac{n+1}{2}  \right\rfloor \leq i+j$, then: $$\begin{array}{lcl}
\phi(\sigma_k\sigma_{k-i}\sigma_j\sigma_{k-i}\sigma_k\sigma_{i+j}\sigma_j\sigma_{i+j}) & = & a(ab)^{n-k}a(ab)^{n+i-k}a(ab)^{n+i-k}a(ab)^{n-k}a(ab)^{n-i-j}a(ab)^{n-i-j}\\
& = & (ba)^{n-k}(ab)^{n+i-k}(ba)^{n+i-k}(ab)^{n-k}(ba)^{n-i-j}(ab)^{n-i-j} = 1.
\end{array}$$
If $\left\lfloor\frac{n+1}{2}  \right\rfloor \leq j$, then:
$$\begin{array}{lcl}
\phi(\sigma_k\sigma_{k-i}\sigma_j\sigma_{k-i}\sigma_k\sigma_{i+j}\sigma_j\sigma_{i+j}) & = & a(ab)^{n-k}a(ab)^{n+i-k}a(ab)^{n-j}a(ab)^{n+i-k}a(ab)^{n-k}a(ab)^{n-i-j}\\
& = & a(ab)^{n-j}a(ab)^{n-i-j}\\
& = & (ba)^{n-k}(ab)^{n+i-k}(ba)^{n-j}(ab)^{n+i-k}(ba)^{n-k}(ab)^{n-i-j}(ba)^{n-j}\\
& = & (ab)^{n-i-j} = 1.
\end{array}$$
Then $\phi$ is a well-defined surjective homomorphism.
\end{proof}

\begin{proof}[Proof of Theorem \ref{Main-Theorem-Quotients}]
The result follows from Theorems \ref{Map-to-D4}, \ref{Map-to-D8}, \ref{Map-to-Dinfty} and Remark \ref{D8-not-needed}.
\end{proof}

\section{Lower central series of cactus groups}\label{Lower-Central-Quotient-Section}
In this section, we make use of the presentation of $J_n$ of Theorem \ref{Minimal-Presentation-Cactus-Rewritten} and the homomorphisms defined in Section \ref{Homomorphisms-Section} to investigate the consecutive quotients of the lower central series of the cactus group. Our goal is to prove Theorems \ref{Lower-Central-Series:DoNotStop/Factors} and \ref{Theorem_J_n/Gamma_3}. We begin by stating some definitions and fundamental results known about the lower central series of a group.\\
The commutator of two elements $g_1$ and $g_2$ of a group $G$ is given by $[g_1,g_2] = g_1^{-1}g_2^{-1}g_1g_2$. Inductively, $$[g_1,g_2, \dots, g_{n}] =[[g_1,g_2, \dots, g_{n-1}], g_{n}]. $$
The lower central series of $G$ is the sequence $\{\Gamma_n(G)\}_{n \in \mathbb{N}}$ given by $$\begin{array}{lcll}
\Gamma_1(G) & = & G, &\\
\Gamma_{n+1}(G)&=&[\Gamma_n(G),G], & \text{ for 
} n\geq 1.
\end{array}$$
Observe that $\Gamma_{n+1}(G) \subset \Gamma_n(G)$, and that $\Gamma_n(G)/\Gamma_{n+1}(G)$ is an Abelian group for every $n\geq 1$. Moreover, it is easy to check that if there exists $n \geq 1$ such that $\Gamma_{n+1}(G)=\Gamma_n(G)$, then for all $i \in \mathbb{N}$, we have $\Gamma_{n+i}(G)=\Gamma_n(G)$. This justifies the definition of nilpotency of a group. A group $G$ is said to be \textit{nilpotent} if there exists $n \geq 1$ such that $\Gamma_n(G)=\{1\}$.\\
We say that the lower central series of a group $G$ {\em stops} at $n$ if $\Gamma_{n+1}(G)=\Gamma_n(G)$ but $\Gamma_{n-1}(G) \neq \Gamma_n(G)$ for $n \geq 1$.
It is obvious that being nilpotent implies that the lower central series stops.\\
 One natural question is to investigate the lower central series of $J_n$. Recall that this problem has been studied in the context of braid groups, see \cite{darne2022lower}.  We prove that the lower central series of $J_n, ~n \geq 3$ does not stop. The case of $J_2 \cong \mathbb{Z}_2$ being finite is not taken into consideration. 

\begin{proof}[Proof of Theorem \ref{Lower-Central-Series:DoNotStop/Factors}]
From Theorem \ref{Minimal-Presentation-Cactus-Rewritten}, we have $J_3 \cong \mathbb{Z}_2 \ast \mathbb{Z}_2$ which is a right-angled Coxeter group. It is not difficult to show that $$\Gamma_n (\mathbb{Z}_2 \ast \mathbb{Z}_2) = \langle (ab)^{2^{n-1}}\rangle$$ for all $n\geq 2$, where $\mathbb{Z}_2 \ast \mathbb{Z}_2 \cong \langle a,b | a^2=b^2=1\rangle$. It then follows that $\mathbb{Z}_2 \ast \mathbb{Z}_2$ is residually nilpotent but not nilpotent. Hence, the lower central series of $J_3$ does not stop. We now consider the case $n \geq 4$ and the new presentation of Theorem \ref{Minimal-Presentation-Cactus-Rewritten}.\\
Consider the group $G = (\mathbb{Z}_2 \ast \mathbb{Z}_2)\times \mathbb{Z}_2$ with the following presentation: $$G=\langle a, b, c ~|~a^2=b^2=c^2=1, ab=ba, ac=ca \rangle.$$ It is well known that the lower central series of a right-angled Coxeter group does not stop. Indeed, as in the previous case, it is not difficult to show that $$\Gamma_n(G)= \langle (cb)^{2^{n-1}} \rangle$$ for all $n \geq 2$.\\
Observe that the following map $$\begin{array}{ccccccccccccc}
\theta & : & J_4 & \longrightarrow & G \\
& & \sigma_2 & \mapsto & a \\
& & \sigma_3 & \mapsto & b \\
& & \sigma_4 & \mapsto & c 
\end{array}$$
is a surjective homomorphism. Indeed, the relations in $J_4$ with the presentation given by Theorem \ref{Minimal-Presentation-Cactus-Rewritten} are of the form: $$\sigma_2^2=\sigma_3^2=\sigma_4^2=(\sigma_4\sigma_2)^4 = (\sigma_4\sigma_3\sigma_2\sigma_3)^2=1$$ and are preserved by $\theta$. Considering the surjective homomorphism $\lambda= q_5 \circ q_6 \circ \dots \circ q_n : J_n \longrightarrow J_4$, which is the composition of maps $ q_i$'s from Theorem \ref{Map-to-D8}, and is given by:
$$\begin{array}{ccccccccc}
\sigma_i & \mapsto & \left\{ \begin{array}{llllllll}
1 & \text{if }  i \leq n-3, \\
\sigma_{i-n+4} & \text{otherwise,}
\end{array}\right.
\end{array}$$ 
we have a surjective homomorphism $\theta \circ \lambda :J_n \to G$ for all $n \geq 4$.\\
Now suppose that there exists $i \geq 1$ such that $\Gamma_i(J_n)/\Gamma_{i+1}(J_n)=\{1\}$. Then the homomorphism $\theta\circ \lambda$ induces a surjective homomorphism $\theta\circ \lambda : \Gamma_i(J_n)/\Gamma_{i+1}(J_n) \to \Gamma_i(G)/\Gamma_{i+1}(G)$ that yields $\Gamma_i(G)/\Gamma_{i+1}(G)=\{1\}$. This implies the lower central series of $G$ stops, a contradiction.
\end{proof}

Now we recall the notion of basic commutators in a group $G$ relative to a given generating set, and the fundamental result by P. Hall regarding the generating set of consecutive quotients of lower central series $\Gamma_i(G)/\Gamma_{i+1}(G)$ via basic commutators. We refer to \cite[Chapter 3]{MR3729243} for a detailed account.\\
Let $G$ be a group generated by the set $X=\{x_1, x_2, \dots, x_k\}$. A \textit{basic commutator} $b_j$ of weight $w(b_j)$ is defined as follows:
\begin{itemize}
\item[(i)] The elements of $X$ are the basic commutators of weight one. We arbitrarily order and relabel them as $b_1, b_2, \dots  ,b_k$ where $b_i < b_j$ if $i < j$.
\item [(ii)] Suppose that we have defined and ordered the basic commutators of weight less than $l > 1$. Then the basic commutators of weight $l$ are $[b_i, b_j]$, where
\begin{itemize}
    \item $b_i$ and $b_j$ are basic commutators and $w(b_i)+w(b_j)=l$; 
    \item $b_i > b_j$; and
    \item if $b_i=[b_s, b_t]$, then $b_j \geq b_t$.
\end{itemize}
\item[(iii)] Basic commutators of weight $l$ come after all basic commutators of weight less than $l$ and are ordered arbitrarily with respect to one another.
\end{itemize}
We now state the basis theorem for the group $\Gamma_i(G)/\Gamma_{i+1}(G)$ which can be found in \cite[Theorem 3.1]{MR3729243}.

\begin{theorem}[P. Hall]\label{Basis-Theorem-Hall}
 Let $G$ be a finitely-generated group with generating set $ \{x_1, x_2, \dots, x_k \}$, and choose $i \in \mathbb{N}$. Then the Abelian group $\Gamma_i(G)/\Gamma_{i+1}(G)$ is generated by the basic commutators of weight $i$. Furthermore, every element of $G$ may be expressed in the (not necessarily unique) form:
$$b_1^{e_1} b_2^{e_2} \dots b_t^{e_t} \Gamma_{i+1}(G),$$
where the $e_j$'s are integers and the $b_j$'s are the basic commutators of weights $1, 2, \dots, i$.
\end{theorem}

We now prove Parts (i) and (ii) of Theorem \ref{Lower-Central-Series:DoNotStop/Factors} by dividing the computations of  $\Gamma_2(J_n)/\Gamma_3(J_n)$ and $\Gamma_3(J_n)/\Gamma_4(J_n)$ into two subsections. We make use of the definition of basic commutators, the new presentation of $J_n$ given in Theorem \ref{Minimal-Presentation-Cactus-Rewritten}, and the homomorphisms given in Section \ref{Homomorphisms-Section}.\\

\subsection{Computation of $\Gamma_2(J_n)/\Gamma_3(J_n)$}\label{Subsection}
We choose the ordering of the generators of $J_n$ as $$\sigma_2 < \ldots < \sigma_n.$$ Then we know that $\Gamma_2(J_n)/\Gamma_3(J_n)$ is generated by the elements $\{[\sigma_i,\sigma_j], ~ n\geq i>j \geq 2 \}$. This set is not minimal: we can reduce the generating set significantly.
\begin{lemma}\label{lemma2}
Suppose that there exists $(i,j,k) \in [2,n]^3$ such that $$(\sigma_{k}\sigma_{i}\sigma_{k}\sigma_{j})^2=1$$
in $J_n$. Then, $[\sigma_{i},\sigma_{j}]\equiv 1 \pmod {\Gamma_3(J_n)}$.
\end{lemma}
\begin{proof}
Let $(i,j,k) \in [2,n]^3$. Then: 
$$\begin{array}{lllll}
& & \sigma_{k}\sigma_{i}\sigma_{k}\sigma_{j}\sigma_{k}\sigma_{i}\sigma_{k}\sigma_{j}=1  \Leftrightarrow  \sigma_{i}\sigma_{k}\sigma_{i}\sigma_{k}\sigma_{j}\sigma_{k}\sigma_{i}\sigma_{k}=\sigma_{i}\sigma_{j} \\& \Leftrightarrow & \sigma_{j}\sigma_{i}\sigma_{k}\sigma_{i}\sigma_{k}\sigma_{j}\sigma_{k}\sigma_{i}\sigma_{k}\sigma_{i}=\sigma_{j}\sigma_{i}\sigma_{j}\sigma_{i} \Leftrightarrow  [\sigma_k, \sigma_i, \sigma_j]=[\sigma_{i},\sigma_{j}].
\end{array}$$
So $[\sigma_i,\sigma_j] \in \Gamma_3(J_n)$, which concludes.
\end{proof}
\begin{corollary}\label{Trivial-Two-Commutators-I}
Let $i+j \leq n$. Then: $$[\sigma_{i},\sigma_{j}] \equiv 1 \pmod {\Gamma_3(J_n)}.$$
\end{corollary}
\begin{proof}
It is an obvious consequence of the relation (\ref{Relation-Two}) in Theorem \ref{Minimal-Presentation-Cactus-Rewritten} and of Lemma \ref{lemma2}.
\end{proof}
\begin{lemma}\label{lemma3}
Let $2 \leq i,j \leq n$. If $i \equiv j \pmod 2$, then: $$[\sigma_{i},\sigma_{j}]\equiv 1 \pmod {\Gamma_3(J_n)}.$$
\end{lemma}
\begin{proof} If $j=i$, then the claim is clearly true. 
Without loss of generality, we assume that $i>j$ and that $i-j=2l$, where  $l \geq 1$. Note that 
$$ 3 \leq i-l=j+l < j \leq n, ~ i-l-j = l \geq 1, \text{ and } j \geq 2,$$ so the triple $(i,i-l,j)$ satisfies the inequalities of relation (\ref{Relation-Three}). Hence, we have the relation: $$\sigma_{i}\sigma_{i-l}\sigma_{j}\sigma_{i-l}=\sigma_{j+l}\sigma_{j}\sigma_{j+l}\sigma_{1,i}.$$
However, $i-l=j+l$, so $(\sigma_{j+l}\sigma_{j}\sigma_{j+l}\sigma_{i})^2=1$, and the result follows from Lemma \ref{lemma2}.
\end{proof}
\begin{lemma}\label{lemma4}
For all $2\leq j \leq i\leq n-2$, we have: $$[\sigma_{i}, \sigma_{j}]\equiv [\sigma_{i+2},\sigma_{j}] \pmod {\Gamma_3(J_n)}.$$
\end{lemma}
\begin{proof}
If $j=i$, then the result holds by Lemma \ref{lemma3}. So assume that $j \neq i,$ $3 \leq i < i+2 \leq n,~ i-j \geq 1 \text{ and } j \geq 2.$\\
By relation (\ref{Relation-Three}), we have
$\sigma_{i+2}\sigma_{i}\sigma_{j}\sigma_{i}=\sigma_{j+2}\sigma_{j}\sigma_{j+2}\sigma_{i+2}$, which is equivalent to: $$\sigma_{i}\sigma_{j}\sigma_{i}\sigma_{i+2}=\sigma_{i+2}\sigma_{j+2}\sigma_{j}\sigma_{j+2}.$$ Therefore, we obtain:
$$\begin{array}{lcl}
[\sigma_{i},\sigma_{j}][\sigma_{i+2},\sigma_{j}]^{-1} &=& \sigma_{i}\sigma_{j}\sigma_{i}\sigma_{j}\sigma_{j}\sigma_{i+2}\sigma_{j}\sigma_{i+2}\\
&=& \sigma_{i+2}\sigma_{j+2}\sigma_{j}\sigma_{j+2}\sigma_{i+2}\sigma_{i+2}\sigma_{j}\sigma_{i+2}\\
&=& \sigma_{i+2}\sigma_{j+2}\sigma_{j}\sigma_{j+2}\sigma_{j}\sigma_{i+2}= \sigma_{i+2}[\sigma_{j+2},\sigma_{j}]\sigma_{i+2}.
\end{array}$$
The result now follows from Lemma \ref{lemma3}.
\end{proof}
\begin{remark}
From Corollary \ref{Trivial-Two-Commutators-I}, for $2i+1\leq n$, we have $[\sigma_{i+1},\sigma_{i}] \equiv 1 \pmod{\Gamma_3(J_n)}$. It is well known that the quotient $\Gamma_2(J_n)/\Gamma_3(J_n)$ is generated by basic commutators of weight $2$, and therefore has the presentation $$\left\langle [\sigma_{i+1},\sigma_{i}], \ \  \left\lfloor  \frac{n+1}{2}  \right\rfloor \leq i \leq n-1 | [\sigma_{i+1},\sigma_{i}]^2=1, \ \ [[\sigma_{i+1},\sigma_{i}],[\sigma_{j+1},\sigma_{j}]]=1\right\rangle.$$
It is easy to check that $\left|\left[ \left\lfloor \frac{n+1}{2}  \right\rfloor,n-1\right]\right|=\left\lfloor \frac{n}{2} \right\rfloor$, so we have shown that $\Gamma_2(J_n)/\Gamma_3(J_n)$ is a subgroup of $\mathbb{Z}_2^{\left\lfloor \frac{n}{2} \right\rfloor}$. We now prove that it is isomorphic to $\mathbb{Z}_2^{\left\lfloor \frac{n}{2} \right\rfloor}$.
\end{remark}

\begin{proof}[Proof of Theorem \ref{Lower-Central-Series:DoNotStop/Factors} (i)]

From Theorem \ref{Map-to-D4}, we have a surjective homomorphism $\varphi : J_n \to D_4$ given by:
$$\begin{array}{cccclcc}

& & \sigma_{i} & \mapsto & \left\{\begin{array}{lll}
1 & \text{if} & i < \left\lfloor \frac{n+1}{2} \right\rfloor \\
a & \text{if} & i \equiv 0 \pmod 2 \ \ \text{and} \ \ i \geq \left\lfloor \frac{n+1}{2} \right\rfloor \\
b & \text{if} & i \equiv 1 \pmod 2 \ \ \text{and} \ \ i \geq \left\lfloor \frac{n+1}{2} \right\rfloor,
\end{array}\right.
\end{array}$$ where $D_4 \cong \langle a,b ~|~a^2=b^2=(ab)^4=1\rangle.$\\

We first recall the lower central series of $D_4$. Since the group $D_4$ is the set $\{1, a, b, ab, aba, (ab)^2, ba, bab \}$ it is easy to check that $\Gamma_2(D_4) = \langle [x,y] ~|~ (x,y) \in D_4^2 \rangle =\{1,(ab)^2\}=\langle (ab)^2 \rangle$ and $\Gamma_3(D_4) = \langle [x,y] ~|~  (x,y) \in \Gamma_2(D_4) \times D_4 \rangle = \{1,[(ab)^2,y], \ y \in D_4\}=\{1\}$.\\

The homomorphism $\varphi$ induces a homomorphism $\widetilde{\varphi} : \Gamma_2(J_n)/\Gamma_3(J_n) \to \Gamma_2(D_4)/\Gamma_3(D_4)$, and for all $i \in [\left\lfloor \frac{n+1}{2} \right\rfloor, n ]$ we have $\widetilde{\varphi}([\sigma_{i+1},\sigma_{i}])=(ab)^2 \neq 1$ in $\Gamma_2(D_4)/\Gamma_3(D_4)$. Therefore, the generators $[\sigma_{i+1},\sigma_{i}]$ for $i \in [\left\lfloor \frac{n+1}{2} \right\rfloor, n ]$ are all non-trivial.\\
It now remains to show that there are no other relations between these generators other than commutation. To see this, it suffices to prove that: $$[\sigma_{i_1+1},\sigma_{i_1}]^{\epsilon_1}\cdots [\sigma_{i_m+1},\sigma_{i_m}]^{\epsilon_m} \not\equiv 1 \pmod{\Gamma_3(J_n)}$$ for any $\left\lfloor \frac{n+1}{2} \right\rfloor\leq i_1,\ldots, i_m \leq n$ and $(\epsilon_1,\ldots \epsilon_m)\in \mathbb{Z}^m$.\\
Suppose on the contrary that there exist $\left\lfloor \frac{n+1}{2} \right\rfloor\leq i_1,\ldots, i_m \leq n$ and $(\epsilon_1,\ldots \epsilon_m)\in \mathbb{Z}^m$ such that $$[\sigma_{i_1+1},\sigma_{i_1}]^{\epsilon_1}\cdots [\sigma_{i_m+1},\sigma_{i_m}]^{\epsilon_m} \equiv 1 \pmod{\Gamma_3(J_n)}.$$ Since, $\Gamma_2(J_n)/\Gamma_3(J_n)$ is an Abelian group and the $[\sigma_{i+1},\sigma_{i}]$ are involutions in the quotient, we may suppose that $i_1 < \ldots < i_m$ and $\epsilon_l=1$ for all $1\leq l \leq m$. Then, there exists $g \in \Gamma_3(J_n)$, such that $[\sigma_{i_1},\sigma_{i_1+1}]\cdots [\sigma_{i_m},\sigma_{i_m+1}]= g$. \\
Since $i_m > i_l \geq \left\lfloor \frac{n+1}{2} \right\rfloor$, for all $1 \leq l \leq m-1$, the canonical inclusion $J_n \subset J_{2i_m}$ is an injective homomorphism \cite{bellingeri2022cactus} and $[\sigma_{i_l+1},\sigma_{i_l}] \equiv 1 \pmod{\Gamma_3(J_{2i_m})}$, for all $1 \leq l \leq m-1$ from Corollary \ref{Trivial-Two-Commutators-I}. Then there exists $\widetilde{g} \in \Gamma_3(J_{2i_m})$ such that $[\sigma_{i_m+1},\sigma_{i_m}]=\widetilde{g}g \equiv 1 \pmod{\Gamma_3(J_{2i_m})}$. But, from what precedes, and since $i_m \geq \left\lfloor \frac{i_m+1}{2} \right\rfloor$, we have $[\sigma_{i_m+1},\sigma_{i_m}] \not\equiv 1 \pmod{\Gamma_3(J_{2i_m})}$, and therefore $[\sigma_{i_1+1},\sigma_{i_1}]\cdots [\sigma_{i_m+1},\sigma_{i_m}] \not\equiv 1\pmod{\Gamma_3(J_n)}$.\\
This concludes the proof of the fact that $\Gamma_2(J_n)/\Gamma_3(J_n) \cong \mathbb{Z}_2^{\left\lfloor \frac{n}{2} \right\rfloor}.$
\end{proof}

\subsection{Computation of $\Gamma_3(J_n)/\Gamma_4(J_n)$}
\begin{proof}[Proof of Theorem \ref{Lower-Central-Series:DoNotStop/Factors}(ii)]
We prove that the elements $[\sigma_{i+1},\sigma_{i},\sigma_{i+1}]$, with $\left\lfloor  \frac{n+1}{2}  \right\rfloor \leq i \leq n-1$, and $[\sigma_{i+1},\sigma_{i},\sigma_{i+2}]$, with $\left\lfloor  \frac{n+1}{2}  \right\rfloor \leq i \leq n-2$, form a minimal generating set of $\Gamma_3(J_n)/\Gamma_4(J_n)$.
From the computations in Subsection \ref{Subsection}, it follows that $\Gamma_3(J_n)/\Gamma_4(J_n)$ is generated by the elements $[[\sigma_{i+1},\sigma_{i}],\sigma_k]$, with $\left\lfloor  \frac{n+1}{2}  \right\rfloor+1 \leq i+1 \leq k \leq n$.\\
From the Jacobi identity in $\Gamma_3(J_n)/\Gamma_4(J_n)$ we have: $$[\sigma_{i+1},\sigma_{i},\sigma_k]\cdot[\sigma_{i},\sigma_k,\sigma_{i+1}]\cdot [\sigma_k,\sigma_{i+1},\sigma_{i}]=1.$$
If $k \equiv i \pmod 2$, then from Lemma \ref{lemma3}, $[\sigma_i,\sigma_k]$ is an element of $\Gamma_3(J_n)$ and so $[[\sigma_i,\sigma_k],\sigma_{i+1}]$ is an element of $\Gamma_{4}(J_n)$. So the above identity becomes: $$[\sigma_{i+1},\sigma_{i},\sigma_k]\cdot[\sigma_{k},\sigma_{i+1},\sigma_{i}]=1.$$
Now, from Lemma \ref{lemma4}, we have $[\sigma_{k},\sigma_{i+1}]= [\sigma_{i+2},\sigma_{i+1}]$, and the Jacobi identity becomes: $$[\sigma_{i+1},\sigma_{i},\sigma_k]\cdot[\sigma_{i+2},\sigma_{i+1},\sigma_i]=1.$$
Finally, since we also have the relation: $$[\sigma_{i+1},\sigma_{i},\sigma_{i+2}]\cdot[\sigma_{i+2},\sigma_{i+1},\sigma_i]=1$$ from the Jacobi identity, we conclude that: $$[\sigma_{i+1},\sigma_{i},\sigma_k]=[\sigma_{i+1},\sigma_{i},\sigma_{i+2}]$$ in $\Gamma_3(J_n)/\Gamma_4(J_n)$.\\
If $k \equiv i+1 \pmod 2$ then in the same way, we prove that: $$[\sigma_{i+1},\sigma_{i},\sigma_k]=[\sigma_{i+1},\sigma_{i},\sigma_{i+1}]$$ in $\Gamma_3(J_n)/\Gamma_4(J_n)$.\\
Finally, observe that: $$[\sigma_{i+1},\sigma_{i},\sigma_{i+1}]=[\sigma_{i+1},\sigma_{i},\sigma_{i}]=(\sigma_i\sigma_{i+1})^4.$$
Then $\Gamma_3(J_n)/\Gamma_4(J_n)$ is generated by the elements $[\sigma_{i+1},\sigma_{i},\sigma_{i+1}]$, with $\left\lfloor  \frac{n+1}{2}  \right\rfloor \leq i \leq n-1$, and $[\sigma_{i+1},\sigma_{i},\sigma_{i+2}]$, with $\left\lfloor  \frac{n+1}{2}  \right\rfloor \leq i \leq n-2$.\\ Now we show the non-trivality of these generators in $\Gamma_3(J_n)/\Gamma_4(J_n)$.\\
By Theorem \ref{Map-to-Dinfty}, there is a surjective homomorphism $\phi: J_n  \to  \mathbb{Z}_2 \ast \mathbb{Z}_2$ given by:
$$\begin{array}{cccclccc}
& & \sigma_i & \mapsto & \left\{\begin{array}{lll}
a(ab)^{n-i} & \text{if} \ \  i \geq \left\lfloor  \frac{n+1}{2}  \right\rfloor\\
1 & \text{otherwise.}
\end{array}\right.
\end{array}$$
Note that for $i \geq \left\lfloor\frac{n+1}{2}  \right\rfloor$, we have:
$$\phi([\sigma_{i+1},\sigma_{i},\sigma_{i+1}]) =\phi([\sigma_{i+1},\sigma_{i},\sigma_{i+2}])=(ba)^4.$$
Since $\Gamma_n(\mathbb{Z}_2 \ast \mathbb{Z}_2)=\langle (ba)^{2^{n-1}} \rangle$, the images of $[\sigma_{i+1},\sigma_{i},\sigma_{i+1}]$ and $[\sigma_{i+1},\sigma_{i},\sigma_{i+2}]$ by $\phi$ are not in $\Gamma_4(\mathbb{Z}_2 \ast \mathbb{Z}_2)$ so that they are non-trivial in $\Gamma_3(J_n)/\Gamma_4(J_n)$.\\
To obtain the desired rank, we prove that these generators are linearly independent.
Suppose on the contrary that there exist $\left\lfloor  \frac{n+1}{2}  \right\rfloor \leq i_1,\ldots,i_m \leq n-1$ and $(\epsilon_1,\epsilon_1',\ldots,\epsilon_m,\epsilon_m')\in \mathbb{Z}^{2m}$, such that: $$[\sigma_{i_1+1},\sigma_{i_1},\sigma_{i_1+1}]^{\epsilon_1}[\sigma_{i_1+1},\sigma_{i_1},\sigma_{i_1+2}]^{\epsilon_1'}\ldots[\sigma_{i_m+1},\sigma_{i_m},\sigma_{i_m+1}]^{\epsilon_m}[\sigma_{i_m+1},\sigma_{i_m},\sigma_{i_m+2}]^{\epsilon_m'} \equiv 1.$$
Since $\Gamma_3(J_n)/\Gamma_4(J_n)$ is an Abelian group and all of its generators are involutions, we may suppose that $i_1 < \ldots < i_m$, and that $\epsilon_j$ and $\epsilon_j'$ are either $0$ or $1$.\\ 
As in the proof of part (i) of Theorem \ref{Lower-Central-Series:DoNotStop/Factors}, we have that $J_n$ injects into $J_{2i_m-1}$, and we see that: $$[\sigma_{i_1+1},\sigma_{i_1},\sigma_{i_1+1}]^{\epsilon_1}=[\sigma_{i_1+1},\sigma_{i_1},\sigma_{i_1+2}]^{\epsilon_1'}= \ldots = [\sigma_{i_m},\sigma_{i_{m-1}},\sigma_{i_{m}+1}]^{\epsilon_{m-1}'}=1$$ in $\Gamma_3(J_{2m-1})/\Gamma_4(J_{2m-1})$.\\
So consider: $$[\sigma_{i_m+1},\sigma_{i_m},\sigma_{i_m+1}]^{\epsilon_m} [\sigma_{i_m+1},\sigma_{i_m},\sigma_{i_m+2}]^{\epsilon_{m'}} \equiv 1 \pmod{\Gamma_4(J_{2i_m-1})}.$$

Now suppose that $\epsilon_m=0$ and $\epsilon_m'=1$. Then $[\sigma_{i_m+1},\sigma_{i_m},\sigma_{i_m+1}]$ belongs to $\Gamma_4(J_{2i_m-1})$. But this is not possible from above.\\
Next we assume that $\epsilon_m=0$ and $\epsilon_m'=1$. Again, it implies that $[\sigma_{i_m+1},\sigma_{i_m},\sigma_{i_m+2}]$ belongs to $\Gamma_4(J_{2i_m-1})$. But again this is not possible from above.\\
Finally, we suppose that $\epsilon_m=1$ and $\epsilon_m'=1$, that is: $$[\sigma_{i_m+1},\sigma_{i_m},\sigma_{i_m+1}] [\sigma_{i_m+1},\sigma_{i_m},\sigma_{i_m+2}]\equiv 1 \pmod{\Gamma_4(J_{2i_m-1})}.$$ We now prove that this does not hold.\\
From Theorem \ref{Map-to-D8}, there is a surjective homomorphism $\psi : J_{2i_m-1}  \to  D_8$ given by:
$$\begin{array}{cccclccccccc}
& & \sigma_l & \mapsto & \left\{\begin{array}{llllll}
a & \text{if} \ \ l=i_m \\
b & \text{if} \ \ l\geq i_m+1 \ \ \text{and} \ \ l \equiv i_m+1 \pmod 2\\
1 & \text{otherwise}.
\end{array}\right.
\end{array}$$
Since $\Gamma_4(D_8)=\{1\}$, we have: $$\psi([\sigma_{i_m+1},\sigma_{i_m},\sigma_{i_m+1}])=[b,a,b]=(ab)^4 \not\equiv 1 \pmod{\Gamma_4(D_8)}$$
and $$\psi([\sigma_{i_m+1},\sigma_{i_m},\sigma_{i_m+2}])=[b,a,1]=1 \pmod{\Gamma_4(D_8)}.$$
Therefore, $[\sigma_{i_m+1},\sigma_{i_m},\sigma_{i_m+1}]$ and $[\sigma_{i_m+1},\sigma_{i_m},\sigma_{i_m+2}]$ do not have the same image in $D_8 / \Gamma_4(D_8)$. Hence, $[\sigma_{i_m+1},\sigma_{i_m},\sigma_{i_m+1}]$ and $[\sigma_{i_m+1},\sigma_{i_m},\sigma_{i_m+2}]$ are distinct in $J_{2i_m-1}$ modulo $\Gamma_4(J_{2i_m-1})$, and therefore in $\Gamma_3(J_{2i_m-1})/\Gamma_4(J_{2i_m-1})$.\\
This concludes the proof of the fact that $\Gamma_3(J_n)/\Gamma_4(J_n) \cong \mathbb{Z}_2^{2\lfloor\frac{n}{2}\rfloor-1}.$
\end{proof}
The table below provides the ranks of small consecutive quotients of the lower central series of $J_n$ computed with GAP using the nq package \cite{nq}. 
\begin{center}\label{table1}
\begin{tabular}{|c||c|c|c|c|c|c|c|c|c|c|c|c|}
  \hline
 $i=$ & $1$ & $2$  & $3$ & $4$ & $5$ & $6$ & $7$ & $8$ & $9$ & $10$ \\
   \hline
      \hline
$J_4$ & $3$  & $2$ & $3$ & $3$ & $4$ & $4$ & $6$ & $7$ & $10$ & $13$ \\

  \hline
$J_5$  & $4$  &$2$  & $3$ & $4$ & $6$ &  $8$ & $12$ & $17$ & $25$ & $36$ \\
  \hline
 $J_6$  & $5$  &$3$  & $4$ & $6$ &  $10$ & $15$ & $26$ & $40$ & $70$ & $114$ \\
  \hline
\end{tabular}\\
\smallskip
The ranks of ${\Gamma_i(J_n)/\Gamma_{i+1}}(J_n)$ for $n=4,5,6$ and $i = 1,2, \dots, 10$.
\end{center}

\subsection{Computation of $J_n/\Gamma_3(J_n)$}
The aim of this subsection is to prove Theorem \ref{Theorem_J_n/Gamma_3}.
\begin{proof}[Proof of Theorem \ref{Theorem_J_n/Gamma_3}]
(i) From the third isomorphism theorem we have the short exact sequence: $$1 \longrightarrow \Gamma_2(J_n)/\Gamma_3(J_n) \longrightarrow J_n/\Gamma_3(J_n) \longrightarrow J_n / \Gamma_2(J_n) \longrightarrow 1.$$
It is known \cite[Proposition 1, p. 139]{MR1472735} that if we have a short exact sequence $$1 \longrightarrow H \longrightarrow G \longrightarrow K \longrightarrow 1$$ with $H=\langle X|R\rangle$ and $K=\langle Y | S \rangle$, then $G=\langle X \sqcup \widetilde{Y}|R\cup \widetilde{S} \cup T\rangle$, where $\widetilde{Y}$, $\widetilde{S}$ and $T$ are defined as follows:
\begin{itemize}
    \item For each $y \in Y$, let $\widetilde{y} \in G$ be a lift of $y$ and let $$\widetilde{Y} = \{\widetilde{y}~|~y\in Y\}.$$
    \item Each $s \in S$ may be written as a word in the elements of $Y$ and their inverses. We may replace each letter $y \in Y$ in $s$ by its chosen lift $ \widetilde{y} \in \widetilde{Y}$, which gives an element $\widetilde{s} \in G$. Since we have a short exact sequence and $s=1$ in $K$, $\widetilde{s}$ belongs to $H$, so we choose a word $w_s$ written in terms of elements of $X$ and their inverses representing $\widetilde{s}$. Then $\widetilde{S}$ is defined as: $$\widetilde{S} = \{\widetilde{s}w_s^{-1}~|~s\in S\}.$$
    \item For each $y \in Y$ and each $x \in X$, the element $\widetilde{y}x\widetilde{y}^{-1}$ is an element of $H$ which may be written as a word $w_{x,y}$ in the elements of $X$ and their inverses. Then $T$ is defined as: $$T = \{\widetilde{y}x\widetilde{y}^{-1}w_{x,y}^{-1}~|~x \in X,~y\in Y\}.$$
\end{itemize}
So by Theorem \ref{Lower-Central-Series:DoNotStop/Factors} (i), and the fact that in $J_n/\Gamma_3(J_n)$, we have: $$\sigma_j[\sigma_{i+1},\sigma_i]\sigma_j = [\sigma_j,\sigma_i][\sigma_{i+1},\sigma_i][\sigma_j,\sigma_i]=[\sigma_{i+1},\sigma_i]$$ for all $\lfloor \frac{n+1}{2}\rfloor \leq i \leq n$ and $2 \leq j \leq n$, we obtain the following: 
$$\begin{array}{lll}
X &=& \left\{[\sigma_{i+1},\sigma_i] ~|~ \lfloor \frac{n+1}{2}\rfloor \leq i \leq n-1 \right\} \\
\widetilde{Y} &=& \{\sigma_i ~|~ 2 \leq i \leq n\}\\
R &=& \left\{ [\sigma_{i+1},\sigma_i]^2,~ ([\sigma_{i+1},\sigma_i][\sigma_{j+1},\sigma_j])^2~|~\lfloor \frac{n+1}{2}\rfloor \leq i,j \leq n-1 \right\} \\
\widetilde{S} &=& \{ \sigma_i^2 \textrm{ for } 2 \leq i \leq n, ~~  \sigma_i\sigma_j\sigma_i\sigma_j[\sigma_{i+1},\sigma_i]~~ \textrm{ if } ~~ j \equiv i+1 \pmod2, ~ \lfloor\frac{n+1}{2}\rfloor \leq i \leq n-1 \textrm{ and } i \leq j \leq n, \\
&& \hspace*{2mm} \sigma_i\sigma_j\sigma_i\sigma_j ~~ \textrm{ if } ~~ j \equiv i \pmod2  \textrm{ or } 2 \leq i < \lfloor\frac{n+1}{2}\rfloor \} \\
T &=& \left\{\sigma_j[\sigma_{i+1},\sigma_i]\sigma_j[\sigma_{i+1},\sigma_i] ~|~ 2 \leq j \leq n \textrm{ and } \lfloor \frac{n+1}{2}\rfloor \leq i \leq n-1 \right\}.
\end{array}$$
Therefore, we can write all the $[\sigma_{i+1},\sigma_i]$ as words in the $\sigma_j$ from the relations in $\widetilde{S}$, the $[\sigma_{i+1},\sigma_i]$ are central by the relations in $T$, and they are involutions by the relations in $R$. Finally, the relations of type (\ref{Relation-Two-Quotient}) and (\ref{Relation-Four-Quotient}) come from the relations in $\widetilde{S}$.\\
For the group $J_4/\Gamma_3(J_4)$, it is not difficult to check that it has the presentation 

$$\langle a_i,~1 \leq i \leq 5 ~|~a_i^2=1,(a_i a_j)^2=1 \text{ for all } i, j \neq 5 \text{ and } a_5 a_1 a_5=a_1 a_3,$$ $$ a_5 a_2 a_5=a_2 a_4, a_5 a_3 a_5=  a_3, a_5 a_4 a_5=a_4 \rangle$$

by the homomorphism given by:
$$a_1 \mapsto \sigma_3 [\sigma_2, \sigma_3] [\sigma_4, \sigma_3], a_2 \mapsto  \sigma_2 [\sigma_3, \sigma_4] [\sigma_2, \sigma_3][\sigma_4, \sigma_3], a_3 \mapsto \sigma_4$$
$$a_4 \mapsto [\sigma_3, \sigma_4] [\sigma_2, \sigma_3][\sigma_4, \sigma_3] \text{ and } a_5 \mapsto [\sigma_3, \sigma_4] [\sigma_2, \sigma_3][\sigma_4, \sigma_3][\sigma_2, \sigma_3][\sigma_4, \sigma_3].$$

Similarly, for the group $J_5/\Gamma_3(J_5)$, we have an equivalent presentation
$$\langle a_i,~1 \leq i \leq 6 ~|~a_i^2=1,(a_i a_j)^2=1 \text{ for all } i, j \neq 5 \text{ and } a_5 a_1 a_5=a_1 a_3, a_5 a_2 a_5=a_2 a_4, a_5 a_3 a_5=  a_3, a_5 a_4 a_5=a_4 \rangle.$$
This concludes the computation of the presentation of $J_n/\Gamma_3(J_n)$.
\end{proof}

\end{document}